\documentclass[10pt]{article}
\usepackage{latexsym,amssymb,amsfonts,euscript}
\usepackage{sidecap}
\usepackage{color}
\usepackage{amsmath, amsthm, empheq} 
\usepackage{palatino} 

\usepackage{wrapfig}
\usepackage{graphicx}
\usepackage{caption}
\usepackage{subcaption}
\newcommand{\bs}{\begin{eqnarray*}}
\newcommand{\es}{\end{eqnarray*}}

\newcommand{\bu}{\bf u}

\newcommand{\bmat}{\begin{pmatrix} }
\newcommand{\emat}{\end{pmatrix}}
\newcommand{\beqs}{\begin{eqnarray*}}
\newcommand{\eeqs}{\end{eqnarray*}}

\newcommand{\beq}{\begin{equation}}
\newcommand{\eeq}{\end{equation}}
\newcommand{\beqa}{\begin{eqnarray}}

\newcommand{\eeqa}{\end{eqnarray}}

\newcommand{\integer}{\hbox{N \kern -.em I}\ }  
\newcommand{\real}{\hbox{R \kern -.2em I}\ }  
\newcommand{\numbersetmm}{\,\hbox{\scriptsize M  \kern -.2em I}\,\ }
\newcommand{\numbersetm}{\,\hbox{M \kern -.2em I}\,\,\ }

\newcommand{\bv}{{\mathbf v}}
\newcommand{\bw}{{\mathbf w}}
\newcommand{\ep}{\epsilon}

\newcommand{\D}{\displaystyle}

\newcommand{\bx}{\mbox{\boldmath $x$}}
\newcommand{\field}[1]{\mathbb{#1}} 

\newcommand{\pr}{\partial}

\newcommand{\bd}{\mbox{\boldmath $d$}}

\newcommand{\be}{\mbox{\boldmath $e$}}
\newcommand{\bU}{\mbox{\boldmath $U$}}
\newcommand{\bK}{\mbox{\boldmath $K$}}

\newcommand{\bq}{\mbox{\boldmath $q$}}

\newcommand{\goto}{\rightarrow}
\newcommand{\B}{\mathbf}
\newcommand{\blue}{\color{blue}}
\newtheorem{theorem}{Theorem}[section]

\begin{document}
\pagestyle{plain}

\begin{center}
{\large
%
On the augmented Biot-JKD equations with Pole-Residue representation of the dynamic tortuosity}%

\end{center}
\begin{center}
{\fontsize{12}{16} \bf 
%
Miao-Jung Yvonne Ou$^\dagger$,\footnote{The work of MYO is partially sponsored by NSF-DMS-1413039} 
Hugo J. Woerdeman$^\ddagger$\footnote{HJW is partially supported by Simons Foundation grant 355645}, 
%
}
\end{center}
\begin{center}
{\fontsize{11}{12} \rm
%
$^\dagger${Department of Mathematical Sciences, University of Delaware, Newark, DE 19716, USA}, mou@udel.edu \\
$^\ddagger${Department of Mathematics, Drexel University, Philadelphia, PA 19104, USA}, hugo@math.drexel.edu
}
\end{center}

\medskip

\centerline{\it Dedicated to our colleague and friend Joe Ball}
 
 \medskip
 
{\fontsize{11}{12}\rm
\abstract
In this paper, we derive the augmented Biot-JKD equations, where the memory terms in the original Biot-JKD equations are dealt with by introducing auxiliary dependent variables. The evolution in time of these new variables are governed by ordinary differential equations whose coefficients can be rigorously computed from the JKD dynamic tortuosity function $T^D(\omega)$ by utilizing its Stieltjes function representation derived in \cite{ou2014on-reconstructi}, where an {\blue approach} for computing the pole-residue representation of the JKD tortuosity is also proposed. The two numerical schemes presented in the current work for computing the poles and residues representation of $T^D(\omega)$ improve the previous scheme in the sense that they interpolate the function at infinite frequency and have much higher accuracy than the one proposed in \cite{ou2014on-reconstructi}. 
\section{\label{intro}Introduction}
{   Poroelastic composites are two-phase composite materials consisted of elastic solid frames with fluid-saturated pore space. The study of poroelasticity plays an important role in biomechanics, seismology and geophysics due to the nature of objects of research in these fields, eg. fluid saturated rocks, sea ice and cancellous bone. It is of great interest for modeling wave propagation in these materials. When the wave length is much higher than the scale of the microstructure of the composite, homogenization theory can be applied to obtain the effective wave equations, in which the fluid and the solid coexist at ever point in the poroelastic material. M. A. Biot derived the governing equations for wave propagation in linear poroelastic composite materials in \cite{biot1956theory-high} and \cite{biot1956theory-of-propa}. The former deals with the low-frequency regime where the friction between the viscous pore fluid and the elastic solid can be assumed to be linear proportional with the difference between the effective pore fluid velocity and the effective solid velocity by a real number $b$, which is independent of frequency $\omega$; this set of equation is referred to as the low-frequency Biot equation. When the frequency is higher than the critical frequency of the poroelastic material, $b$ will be frequency-dependent; this is the subject of study in \cite{biot1956theory-of-propa}.  The exact form of $b$ as a function of frequency was derived in \cite{biot1956theory-of-propa} for pore space with its micro-geometry being circular tubes. A more general expression was derived in the seminal paper \cite{johnson1987theory-of-dynam} by Johnson, Koplik and Dashen (JKD), where causality argument was applied to derive the 'simplest' form of $b$ as a function of frequency. This frequency dependence of $b$ results in a time-convolution term in the time-domain poroelastic wave equations; the kernel in the time-convolution term is called the 'dynamic tortuosity' in the literature. The Biot-JKD equations refer to the Biot equations with $b$ being the JKD-tortuosity in \eqref{tortuosity-JKD}.

In general, the dynamic tortuosity function is a tenor, which is related to the symmetric, positive definite dynamic permeability tensor of the poroelastic material $\B{K}(\omega)$ by 
$T(\omega)=\frac{i \eta \phi}{\omega \rho_f} \B{K}^{-1}(\omega)$. By the definition of dynamic tortuosity and dynamic permeability it is clear that their principal directions coincide. In the principal direction $x_j$, $j=1,\cdots, 3$ of $\B{K}$, the JKD tortuosity is
\beqa
T^J_j(\omega)=\alpha_{\infty j}\left(  1-\frac{\eta\phi}{i\omega\alpha_{\infty j}\rho_{f} K_{0j}}\sqrt{1-i\frac{4\alpha_{\infty j}^{2}K_{0j}^{2}\rho_{f}\omega}{\eta\Lambda_j^{2}\phi^{2}}}\right),\, j=1,2,3 \label{tortuosity-JKD}
\eeqa
with the tunable geometry-dependent constant $\Lambda_j$,  
the dynamic viscosity of pore fluid $\eta=\rho_f\,\nu$, the porosity $\phi$, the fluid density $\rho_f$, the static permeability $K_j$ and the infinite-frequency-tortuosity $\alpha_{\infty j}$; all of these parameters are positive real numbers. We refer to $T^J_j(\omega)$ as the JKD tortuosity function in the $j$-th direction..

The time-domain low-frequency Biot's equations have been numerically solved by many authors. However, for high-frequency Biot equations such as Biot-JKD, the time convolution term remains a challenge for numerical simulation. Masson and Pride \cite{Masson} defined a time convolution product to discretize the fractional derivative. Lu and Hanyga \cite{Lu} developed a new method to calculate the shifted fractional derivative without storing and integrating the entire velocity histories. Recent years, Chiavassa and Lombard et al. \cite{Blance, Lombard2} used an optimization procedure to approach the fractional derivative.

In this work, we will derive an equivalent system of the Biot-JKD equations without resorting to the fractional derivative technique. The advantage of this approach is that the new system of equations have the same structure as the low-frequency Biot equations but with more variables. Hence we refer to this system as the augmented Biot-JKD equations. A key step in this derivation is to utilize the Stieltjes function structure to compute from the given JKD tortuosity the coefficient of the additional terms.}

The first-order formulation of the time-domain Biot equations consists of the strain-stress relations of the poroelastic materials and the equation of motions. The solid displacement $\bu$, the pore fluid velocity relative to the solid $\bq$ and the pore pressure $p$ are the unknowns to be solved.  In terms of the solid displacement $\bu$, we define the following variables
\[
\bv:= \partial_t \bu \mbox{ (solid velocity) }, \bw:= \phi(\bU-\bu) \mbox{(fluid displacement relative to the solid) }, \bq:=\partial_t \bw, \zeta:=-\nabla\cdot \bw
\]
where $\phi$ is the porosity. Here $\bU$ is the averaged fluid velocity over a representative volume element. The spatial coordinates $(x_1,x_2,x_3)$ are chosen to be aligned with the principal directions of the static permeability tensor $\bK$ of the poroelastic material, which is know to be symmetric and positive definite. 

Let $\ep_{ij}:=\frac{1}{2}(\frac{\partial u_i}{\partial x_j}+\frac{\partial u_j}{\partial x_i})$ be the linear strain of the solid part, then the stress-strain relation  is given by \cite{biot1962mechanics-of-de}
\beq
\begin{pmatrix}
\sigma_{11}\\ \sigma_{22}\\ \sigma_{33} \\ \sigma_{23}\\ \sigma_{13}\\ \sigma_{12}\\ p
\end{pmatrix}
=\begin{pmatrix}
c_{11}^u & c_{12}^u & c_{13}^u & c_{14}^u & c_{15}^u & c_{16}^u& M\alpha_1\\
c_{12}^u & c_{22}^u & c_{23}^u & c_{24}^u & c_{25}^u & c_{26}^u & M\alpha_2\\
c_{13}^u & c_{23}^u & c_{33}^u & c_{34}^u & c_{35}^u & c_{36}^u & M\alpha_3\\
c_{14}^u & c_{24}^u & c_{34}^u & c_{44}^u & c_{45}^u & c_{46}^u & M\alpha_4\\
c_{15}^u & c_{25}^u & c_{35}^u & c_{45}^u & c_{55}^u & c_{56}^u & M\alpha_5\\
c_{16}^u & c_{26}^u & c_{36}^u & c_{46}^u & c_{56}^u & c_{66}^u & M\alpha_6\\
M\alpha_1 & M\alpha_2 & M\alpha_3 & M\alpha_4 & M\alpha_5 & M\alpha_6 & M   
\end{pmatrix}
\begin{pmatrix}
\ep_{11} \\ \ep_{22} \\ \ep_{33} \\ 2\ep_{23}\\ 2\ep_{13} \\ 2\ep_{12} \\ -\zeta
\end{pmatrix}
\eeq
where $p$ is the pore pressure, $c_{ij}^u$ are the elastic constants of the undrained frame, which are related to the elastic constants $c_{ij}$ of the drained frame by $c_{ij}^u=c_{ij}+M a_i a_j$, $i,j=1,\ldots,6$. In terms of the material bulk moduli $\kappa_s$ and $\kappa_f$ of the solid and the fluid, respectively, the fluid-solid coupling constants $a_i$ and $M$ are given by
\beqs
a_i:=\begin{cases}
1-\frac{1}{3\kappa_s}\sum_{k=1}^{3} c_{ik} & \mbox{ for } i=1,2,3,\\
-\frac{1}{3\kappa_s}\sum_{k=1}^{3} c_{ki} & \mbox{ for } i=4,5,6,
\end{cases}\\
M:=\frac{\kappa_s}{1-\overline{\kappa}/\kappa_s - \phi(1-\kappa_s/\kappa_f)} ,\\
\overline{\kappa}:=\frac{c_{11}+c_{22}+c_{33}+2c_{12}+2c_{13}+2c_{23}}{9} .
\eeqs 

The six equations of motion are as follows
\beqa
\sum_{k=1}^3 \frac{\partial \sigma_{jk}}{\partial x_k}=\rho \frac{\partial v_j}{\partial t} + \rho_f \frac{\partial q_j}{\partial t},\, t>0 , \label{test}\\
-\frac{\partial p}{\partial x_j} = \rho_f \frac{\partial v_j}{\partial t}+ \left(\frac{\rho_f}{\phi} \right) \check{\alpha_j} \star \frac{\pr q_j }{\pr t},\,t>0,\, j=1,2,3,
\label{convo}
\eeqa
where $\star$ denotes the time-convolution operator,  $\rho_f$ and $\rho_s$ are the density of the pore fluid and of the solid, respectively, $\rho:=\rho_s(1-\phi)+\phi \rho_f$ and $\check{\alpha_j}$ is the inverse Laplace transform of the dynamic tortuosity $\alpha_j(\omega)$ with $\omega$ being the frequency. Here the one-sided  Laplace transform of a function $f(t)$ is defined as
\[
\hat{f}(\omega):={\cal{L}}[f](s=-i\omega):={   \int_{0}^{\infty} f(t)e^{-st} dt}.
\]

As a special case of the Biot-JKD equations, the low frequency Biot's equation corresponds to 
\[
\hat{\alpha_j}(t)=\alpha_{\infty j} \delta(t)+\frac{\eta \phi}{K_{0j}\rho_f}H(t) ,
\] 
where $\delta(t)$ is the Dirac function and $H(t)$ the Heaviside function, 
$\eta$ the dynamic viscosity of the pore fluid, $K_{0j}$ the static permeability in the $x_j$ direction. This low-frequency tortuosity function corresponds to 
\[
\alpha_j(\omega)=\alpha_{\infty j}+\frac{\eta \phi/K_{0j} \rho_f}{-i\omega} .
\]

{   In the Biot-JKD equation, we have $\alpha_j(\omega)=T^J_j(\omega)$.}

According to Theorem 5.1 in \cite{ou2014on-reconstructi}, in the principal coordinates $\{x_j\}_{j=1}^3$ {   of the permeability tensor $\B{K}$} and for $\omega$ such that $-\frac{i}{\omega} \in \field{C}\setminus \theta_1$,  the {   JKD} dynamic tortuosity function has the following integral representation formula
\beqa
T^J_j(\omega)=a_j\left(\frac{i}{\omega}\right)+\int_0^{\theta_1} \frac{d\sigma_j(t)}{1-i\omega t},\,\,   a_j:=\frac{\eta \phi}{\rho_f K_{0j}},\, j=1,2,3,
\label{Tortuosity_IRF}
\eeqa
where $0<\theta_1<\infty$  and the positive measure $d\sigma_j$ has a Dirac measure of strength $\alpha_{\infty j}$ sitting at $t=0$; this is to take into account the asymptotic behavior of dynamic tourturosity  as frequency goes to $\infty$. This function is the analytic continuation of the usual dynamic tortuosity function  in which $\omega\ge 0$.  As a function of the new variable $s:=-i\omega$, $\omega\in \field{C}$,  the singularities of \eqref{Tortuosity_IRF} are included in the interval $(-\infty,-\frac{1}{\theta_1})$ and a simple pole sitting at $s=0$. Therefore, if we define a new function for each $j=1,2,3$
\beq
D^J_j(s):=T^J_j(\omega)-\frac{i a_j}{\omega}=\int_0^{\theta_1} \frac{d\sigma_j(t)}{1+s t} ,
\label{IRF-Ds}
\eeq
then $D^J_j(s)$ is analytic in $\field{C}\setminus (-\infty,-\frac{1}{\theta_1})$ on the s-plane. This type of functions are closely related to the well-known Stieltjes functions. The first {   approach} we propose in this paper is based on the fact \cite{gelfgren1981multipoint} that a Stieltjes function can be well approximated by its Pad\'{e} approximant whose poles are all simple. The other {   approach} proposed here for computing the pole-residue approximation of the dynamic tortuosity function is based on the result in \cite{ALPAY1994485}.

We note that  
\[
D^J_j(s)=\alpha_{\infty j}\left(  1+\frac{\eta\phi}{s\alpha_{\infty j}\rho_{f} K_{0j}}\sqrt{1+s\frac{4\alpha_{\infty j}^{2}K_{0j}^{2}\rho_{f}}{\eta\Lambda_j^{2}\phi^{2}}}\right)-\frac{a_j}{s}=:\alpha^J(s)-\frac{a_j}{s}, j=1,2,3
\]
is analytic away from the {\bf branch cut} on $[0,C_1]$ along the real axis, where $C_1:=\frac{4 \alpha_{\infty j}^2 K_{0j}^2}{\nu\phi^2 \Lambda_j^2}$. Therefore
\beq
D^J_j(s)=\int_0^{   C_1} \frac{d\sigma^J_j(t)}{1+st}\approx \alpha_\infty+ \sum_{k=1}^M \frac{r_k^j}{s-p_k^j} \mbox{ for } s\in \field{C}\setminus (-\infty,-\frac{1}{C_1}], j=1,2,3.
\label{DJ_approx}
\eeq
with $r_k^j>0$, $p_k^j<-\frac{1}{C_1}<0$, $j=1,2,3$, $k=1,\ldots,M$, that can be computed from dynamic permeability data $K_{   j}(\omega)$ evaluated at $M$ different frequencies in the frequency content of the initial waves. The special choice of $s=-i\omega$, $\omega\in \field{R}$ in \eqref{DJ_approx} provides a pole-residue approximation of $T^J_{   j}(\omega)$, $j=1,2,3$.

Applying Laplace transform to the convolution term in \eqref{convo} with JKD tortuosity, i.e. $\alpha=\alpha^J$, (see eg.  Theorem 9.2.7 in \cite{ACV_Dettman})
\beqa
{\cal{L}}[\check{\alpha^J}\star \frac{\pr q_j}{\pr t}](s)&=& \alpha^J(s)(s \hat{q_j})\\
&=&\left(D^J(s)+\frac{a}{s}\right)(s\hat{q_j})\\
&\approx& \left( \alpha_\infty+\sum_{k=1}^M \frac{r_k}{s-p_k} +\frac{a}{s}  \right) (s \hat{q_j})\\
&\approx& \alpha_\infty s \hat{q_j}+\sum_{k=1}^M \frac{r_k}{s-p_k}  (s \hat{q_j})+ a\hat{q_j}\\
&=& \alpha_\infty s \hat{q_j}+\left(a+\sum_{k=1}^M  r_k \right) \hat{q_j} +\sum_{k=1}^M  r_k p_k \frac{\hat{q_j}}{s-p_k} .
\label{Lap_trans}
\eeqa
Notice that
\beqa
s\hat{q_j}={\cal{L}}\left[ \pr_t q^j\right]+q_j(0).
\eeqa
Furthermore, for each of the terms in the sum, since all the singularities $p_k$ are restricted to the left of $s=-\frac{1}{C_1}$, the inverse Laplace transform can be performed by integrating along the imaginary axis for $t>0$, i.e., 
\[
{\cal{L}}^{-1}\left[\frac{1}{s-p_k}\right](t)=  {   \frac{1}{2\pi i} } \int_{-i \infty}^{i\infty}\frac{1}{\zeta-p_k} e^{\zeta t}d\zeta=  r_k e^{p_k t},\, t>0.
\]
This integral  is evaluated by integrating  along $[-Ri, Ri]\cup \{s=Re^{i\theta} |  \pi/2< \theta < 3\pi/2\}$ and applying the residue theorem and letting $R\goto \infty$. As a result, we have for $t>0$
\beqa
\left(\check{\alpha^J}\star \frac{\pr q_j}{\pr t}\right) (\bx,t)&:=&\int_0^t \check{\alpha^J}(\tau)\frac{\pr q_j}{\pr t}(\bx,t-\tau)d\tau\\
&\approx& \alpha_\infty \left(\frac{\pr q_j}{\pr t}\right)+ \left(a+ \sum_{k=1}^M  r_k \right) q_j
-\sum_{k=1}^M  r_k (-p_k) e^{p_k t} \star  q_j.
\eeqa

Applying a strategy similar to those in the literature \cite{carcione2001wave-fields-in-}, we define the auxiliary variables $\Theta_k$, $k=1,\ldots,M$ such that
\beq
\Theta_k^{x_j}(\bx,t):= (-p_k) e^{p_k t} \star  q_j.
\label{theta_def}
\eeq
It can be easily checked that $\Theta_k$, $k=1,\ldots,M$, satisfies the following equation:
\beq
\pr_t\Theta_k^{x_j}(\bx,t)=p_k \Theta_k^{x_j}(\bx,t)-p_k q_j(\bx,t).
\eeq

We assume $\B{q}(0)=0$ for simplicity. For an anisotropic media, each principal direction $x_j$, $j=1,2,3$, has a different tortuosity function $\alpha_j$. We label the corresponding poles and residue  as $p_k^{x_j}$ and $r_k^{x_j}$ and modify \eqref{theta_def} accordingly. Replacing the convolution terms in \eqref{convo} with the equations of $\Theta
_k^{x_j}$, we obtain the following system that has no explicit memory terms:
\begin{empheq}[left=\empheqlbrace]{align}
\sum_{k=1}^3 \frac{\partial \sigma_{jk}}{\partial x_k}& =\rho \frac{\partial v_j}{\partial t} + \rho_f \frac{\partial q_j}{\partial t},\, t>0  ,
\label{aug_b}
\\
\pr_t\Theta_k^{x_j}(\bx,t)&=p_k \Theta_k^{x_j}(\bx,t)-p_k q_j(\bx,t),\, j=1,2,3,\\
-\frac{\partial p}{\partial x_j} &= \rho_f \frac{\partial v_j}{\partial t}+ \left(\frac{\rho_f \alpha_{\infty j}}{\phi} \right) \frac{\pr q_j }{\pr t}
  + \left( \frac{\eta}{K_{0j}}+\frac{\rho_f}{\phi}\sum_{k=1}^M r_k \right) q_j \nonumber\\
 &-\left(\frac{\rho_f}{\phi}\right)\sum_{k=1}^M r_k \Theta_k^{x_j},\,t>0,\, j=1,2,3. 
 \label{aug_e}
\end{empheq}
We refer to this system as the augmented system of Biot-JKD equations in the principal directions {   of the permeability tensor $\B{K}$}. 
\section{\label{algo}Numerical scheme for computing $r_k$ and $p_k$}
Since the function $D^J$ results from subtracting the pole of $T^J$ at $s=0$, it has a removable singularity at $s=0$ and is analytic away from its branch-cut located at $(-\infty,-1/C_1]$. Both {   a{   Approach}pproach}s presented here are based on the fact that $D^J(s)$ is a Stieltjes function. 

The problem to be solved is formulated as follows. Given the data of $D^J$ at distinct values of $s=s_1,\ldots, s_M$, construct the pole-residue approximation of $D^J$ such that 
\beq
D^J(s)\approx D^J_{est}(s):=\alpha_\infty+\sum_{k=1}^M \frac{r_k}{s-p_k} \mbox{ for } s\in[s_1,s_M] \mbox { and } r_k>0, \, p_k<0, \ \forall k=1,\dots ,M.
\label{the_point}
\eeq
\subsection{Rational function approximation and partial fraction decomposition}
This approximation takes into account the asymptotic behavior $\lim_{s\goto \infty}=\alpha_\infty$ and hence can be considered as an improved version of the reconstruction {   Approach} for tortuosity in \cite{ou2014on-reconstructi}, which does not interpolate at infinity. In this paper, we also take into account the asymptotic behaviors of $D(s)$. Note that 
\beq
\lim_{\omega\goto 0^+}D(s=-i\omega)=\alpha_\infty+ 2\left(\frac{\alpha_\infty}{\Lambda}\right)^2\frac{K_0}{\phi},\hspace{0.5in}
\lim_{\omega\goto \infty}D(s=-i\omega)= \alpha_\infty
\label{asymp_D}
\eeq
By a theorem in \cite{gelfgren1981multipoint}, we know that the poles in the Pad\'{e} approximant of $D^J(s)$ have to be contained in $(-\infty,-1/C_1]$ and are all simple with positive weight (residue), this implies that the constant term in the denominator in the Pad\'{e} approximant can be normalized to one. According to the aforementioned theorem, if $(s,D^J(s))$ is an interpolation point with $Im(s)\ne 0$, then $(\overline{s}, D^J(\overline{s}) )$ must also be an interpolation point, where $\overline{\cdot}$ represents the complex conjugate. From the integral representation formula (IRF), we know that $D^J(\overline{s_k})=\overline{D^J(s_k)}$.

Hence, the following approximation problem is considered:
Given $M$ data points $D^J(s_k=-i\omega_k)\in \field{C}$, $k=1,\ldots,M$, find $\bx:=(a_0,\cdots,a_{M-1}, b_1,\cdots,b_{M})^t$ such that 
$$
(S)
\begin{cases}
\D{
D^J(s_k)-\alpha_\infty = \frac{a_0+a_1 s_k+\cdots+a_{M-1}s_k^{M-1}} {1+b_1 s_k+\cdots+b_{M}s_k^{M}} },k=1,\ldots,M,\\
\D{
\overline{D^J(s_k)}-\alpha_\infty = \frac{a_0+a_1 \overline{s_k}+\cdots+a_{M-1}\overline{s_k}^{M-1}} {1+b_1 \overline{s_k}+\cdots+b_{M}\overline{s_k}^{M}},\, k=1,\ldots,M ,}
\end{cases}
$$
where $\omega_k$, $k=1,\ldots,M$, are distinct positive numbers. However, with a closer look, this system of equations is equivalent to the one by enforcing the condition $\bx\in \field{R}^{2M}$ to the first half of $(S)$. To be more specific, we define
\beqa
A &:=&
\begin{pmatrix}
1 & s_1 &s_1^2 &\cdots & s_1^{M-1} & -D^J(s_1) s_1 & - D^J(s_1) s_1^2 &\cdots & -D^J(s_1) s_1^{M}\\
1 & s_2 &s_2^2 &\cdots & s_2^{M-1} & -D^J(s_2) s_2 & - D^J(s_2) s_2^2 &\cdots & -D^J(s_2) s_2^{M}\\
\vdots &\vdots& \vdots& &\vdots& \vdots&\vdots & &\vdots\\
1 & s_M &s_M^2 &\cdots & s_M^{M-1} & -D^J(s_M) s_M & - D^J(s_M) s_M^2 &\cdots & -D^J(s_M) s_M^{M}
\end{pmatrix} \in \field{C}^{M\times 2M} , \nonumber\\ 
\bd&:=&(D^J(s_1){   -\alpha_\infty},\,   D^J(s_2) {   -\alpha_\infty},\ \cdots\  D^J(s_{M}){   -\alpha_\infty})^t \in \field{C}^{M},  \nonumber\\
\bx&:=&(a_0, \cdots, a_{M-1}, b_1,\cdots, b_M)^t \in \field{R}^{M}.
\eeqa
Then the system to be solved is
\beq
\begin{pmatrix}
Re(A)\\Im(A)
\end{pmatrix}
\bx=\begin{pmatrix} Re(\bd)\\Im(\bd) \end{pmatrix}, 
\label{Algo_1}
\eeq
where $Re()$ and $Im()$ denote the real part and the imaginary part, respectively. After solving for $\bx$, the poles and residues are then obtained by the partial fraction decomposition of the Pad\'{e} approximant, i.e.
\beq
\frac{a_0+a_1 s+\cdots+a_M s^{M-1}}{1+b_1s+\cdots+b_M s^M}=\sum_{j=1}^{M} \frac{r_j}{s-p_j}.
\eeq

\subsection{Two-sided residue interpolation in the Stieltjes class}
The second {   approach} is based on the following theorem that can be considered as a special case of what is proved in \cite{ALPAY1994485}. The advantage of this method is that it explicitly identifies the poles $p_k$, $k=1,\ldots,M$ as the generalized eigenvalues of matrices constructed from the data. We note that the interpolation problem below also appears in the recent paper \cite{ALI}, where the main focus is model reduction.

Let $\field{C}^+:=\{z\in \field{C}: Im(z)>0\}$. Given $M$ interpolation data $(z_i,u_i,v_i)\in \field{C}^+\times \field{C}^{p\times q}\times \field{C}^{p\times q}$, we seek a $p\times p$ matrix valued function $F(z)$ of the form
\beq
F(z)=\int_0^\infty \frac{d\mu(t)}{t-z}, \mbox{ where $\mu$ is a positive $p\times p$ matrix -valued measure}
\label{F1}
\eeq
 such that 
 \beq
 F(z_i)u_i=v_i, i=1, \ldots , M.
 \label{F2}
 \eeq
 \begin{theorem}
 \label{two-sided}
 If there exists a solution $F(z)$ described as above, then the Hermitian matrices $S_1$ and $S_2$ defined via
 \beq
 (S_1)_{ij}=\frac{u_i^*v_j - v_i^*u_j}{z_j-\overline{z_i}},\, (S_2)_{ij}:=\frac{z_j u_i^* v_j-\overline{z_i}v_i^*u_j}{z_j-\overline{z_i}},\, i,j=1,\ldots,M,
 \eeq 
 are positive semidefinite. Conversely, if $S_1$ is positive definite and $S_2$ is positive semidefinite,  then 
 \beq
 F(z):=-C_+(zS_1-S_1 A -C_+^*C_-)^{-1}C_+^* =C_+(S_2-zS_1)^{-1}C_+^*
 \eeq
 is a solution to the interpolation problem. Here
 \[
 C_-:= \begin{pmatrix} u_1 & \cdots & u_M \end{pmatrix},\, C_+:= \begin{pmatrix} v_1 & \cdots & v_M \end{pmatrix},\, A:={\rm diag}(z_i I_q)_{i=1}^M,
 \]
 and $I_q$ is the identity matrix of dimension $q$.
 \end{theorem}
\begin{proof} 
Suppose \eqref{F1} and \eqref{F2} are true. Then we have 
\[
u_i^* v_j -v_i^* u_j=u_i^*(F(z_j)-F(z_i)^*)u_j=(z_j-\overline{z_i}) u_i^*\left(\int_0^\infty \frac{d\mu(t)}{(t-z_j)(t-\overline{z_i})}\right)u_j.
\]
Thus 
\[
S_1=\int_0^\infty \begin{pmatrix} \frac{u_i^*}{t-\overline{z_1}}\\ \vdots \\ \frac{u_M^*}{t-\overline{z_M}}\end{pmatrix} d\mu(t) 
\begin{pmatrix} \frac{u_1}{t-z_1} & \cdots & \frac{u_M}{t-z_M} \end{pmatrix} \ge 0,
\]
\[
S_2=\int_0^\infty \begin{pmatrix} \frac{u_i^*}{t-\overline{z_1}}\\ \vdots \\ \frac{u_M^*}{t-\overline{z_M}}\end{pmatrix} \, t d\mu(t) 
\begin{pmatrix} \frac{u_1}{t-z_1} & \cdots & \frac{u_M}{t-z_M} \end{pmatrix} \ge 0 .
\]
Conversely, suppose $S_1>0$ and $S_2 \ge 0$. Notice that 
\begin{empheq}[left=\empheqlbrace]{align}
&A^*S_1-S_1 A = C_+^* C_--C_-^*C_+ , \label{S1a}\\
&A^*S_2-S_2 A = A^*C_+^* C_--C_-^*C_+A .\label{S2a}
\end{empheq}
These equations uniquely determine $S_1$ and $S_2$ as the spectra of $A$ and $A^*$ do not overlap. Observe that if $S_1$ satisfies (\ref{S1a}), then $S_2:=S_1A+C_+^* C_-$ is the solution of (\ref{S2a}). Therefore, we have
\beq
S_2=S_1A+C_+^* C_- .
\eeq
Note that $S_2-zS_1=S_1^{\frac{1}{2}} (S_1^{-\frac{1}{2}} S_2 S_1^{-\frac{1}{2}}-z) S_1^{\frac{1}{2}}$. Since $S_1^{-\frac{1}{2}} S_2 S_1^{-\frac{1}{2}}$ has eigen values in $[0,\infty )$, $(S_1^{-\frac{1}{2}} S_2 S_1^{-\frac{1}{2}}-z)$ is invertible for $z\notin [0,\infty)$.  Let $(X,D)$ be the eigen decomposition such that 
\[
S_1^{-\frac{1}{2}} S_2 S_1^{-\frac{1}{2}}=XDX^* \mbox{ with } X=\begin{pmatrix} \bx_1 & \cdots & \bx_{qM}
\end{pmatrix},\, D={\rm diag}(d_j)_{j=1}^{qM}.
\]
Then we have for  $z\notin [0,\infty)$
\[
F(z)=\sum_{j=1}^{qM}\left( \frac{1}{d_j-z} \right)C_+ S_1^{-\frac{1}{2}} \bx_j \bx_j^* S_1^{-\frac{1}{2}} C_+^*,
\]
and thus $F(z)$ has the required form with $d\mu$ being a atomic measure supported on $d_1,\ldots,d_{qM}$.

Furthermore, letting $\be_1,\ldots,\be_M$ be the standard basis vectors of $\field{R}^M$, we have for $i=1,\ldots,M$,
\[
(z_i S_1-S_1 A-C_+^*C_-)(\be_i \otimes I_q)=S_1(z_i-I-A)(\be_i\otimes I_q)-C_+^*C_-(\be_i\otimes I_q)=0-C^*_+u_i=-C^*_+u_i .
\]
Thus
\[ 
F(z_i)u_i=-C_+(z_i S_1-S_1 A-C_+^*C_-)^{-1}C_+^*u_i=-C_+(-\be_i\otimes I_q)=v_i.
\]
\end{proof}
To apply this theorem to our problem, we first note that if we identify $z$ in Theorem \ref{two-sided} with $-\frac{1}{s}$, then the IRF for $D^J_j(s)$ in \eqref{DJ_approx}, denoted by $D^J$ for simplicity, can be written as 
\[
D^J(s)=(-z) \int_0^{\Theta_1}\frac{d\sigma^J}{t-z},
\]
and 
\beq
D^J(s)-\alpha_\infty=(-z)\left(\int_0^{\Theta_1}   \frac{d\sigma^J(t)}{t-z} -\frac{\alpha_\infty}{-z}   \right)= (-z)\left(\int_0^{\Theta_1}   \frac{d\sigma^J(t)}{t-z} -\int_0^{\Theta_1}\frac{\alpha_\infty \sigma(t)}{t-z}   \right) ,
\label{DJm}
\eeq
where $\sigma(t)$ is a Dirac measure at $t=0$. Since $\sigma^J$ has a Dirac measure of strength $\alpha_\infty$, the function inside the parentheses in \eqref{DJm} is a Stieltjes function, which we denote by $F_{new}(z)$, i.e.
\[
D^J(s)-\alpha_\infty=(-z)F_{new}(z)
\]  

What we would like to harvest is the pole-residue approximation of $D(s)-\alpha_\infty$. To avoid truncation error, we rewrite all the formulas in Theorem \ref{two-sided} in terms of variable $s=-\frac{1}{z}$ as follows.

\beqa
s_i=-\frac{1}{z_i},\,
u_i=\frac{1}{s_i},\\
v_i=D(s_i)-\alpha_\infty,\, i=1\ldots M,\\
(S_1)_{ij}=\frac{-s_j D(s_j)+s_i^* D^*(s_i)}{s_i^*-s_j},\\
(S_2)_{ij}=\frac{-D(s_j)+D^*(s_i)}{s_j-s_i^*}.
\eeqa
Consequently, we have the following representation for $D(s)$
\beq
D^J(s)\equiv \alpha_\infty+{  (\frac{1}{s})} F_{new}(-\frac{1}{s})=\alpha_\infty+\sum_{j=1}^{qM}\left( \frac{   1}{sd_j+1} \right)C_+ S_1^{-\frac{1}{2}} \bx_j \bx_j^* S_1^{-\frac{1}{2}} C_+^*.
\eeq
With the generalized eigenvalues $[\B{V},\B{L}]:=eig(S_1,S_2)$, where $\B{V}$ is the matrix of generalized  vectors and $\B{L}$ the diagonal matrix of generalized eigenvalues such that 
\beq
S_1 \B{V}=S_2 \B{V} \B{L},
\label{Algo_2}
\eeq 
and taking into account of the simultaneous diagonalization property
\beqa
\B{V}^* \B{S}_1 \B{V}=\B{L}\\
\B{V}^* \B{S}_2 \B{V}=\B{I}
\eeqa
we have
\beq
D^J(s)=\alpha_\infty+\sum_{k=1}^{N} \frac{C_+ \B{V}(:,k)\B{V}(:,k)^* C_+^*}{s+{L}(k,k)}.
\label{two-sided}
\eeq

The poles $\vartheta_k$ and residues $r_k$ are given by
\begin{eqnarray}
  p_k &=& - L(k,k) \\
  r_k &=&C_+ \B{V}(:,k)\B{V}(:,k)^* C_+^*
  \end{eqnarray}

%
\section{Numerical Examples}
In this section, we apply both {   approach}s in Section \ref{algo} to the examples of cancellous bone (S1) studied in \cite{hosokawa2006ultrasonic-puls}, \cite{fellah2013transient-ultra} and the epoxy-glass mixture (S2 and S3) and the sandstone (S4 and S5) examples studied in \cite{blanc2014wave-simulation}. From prior results, it is known that wider range the frequency is, the more ill-conditioned the corresponding matrices will be. We focus on the test case in \cite{blanc2014wave-simulation}, which apply the fractional derivate approach to deal with the memory term.  In this case, time profile of the source term, denoted by $g(t)$ is a Ricker signal of central frequency $f_0=10^5\,s^{-1}$ and time-shift $t_0=1/f_0$, i.e.
\[
g(t)=\begin{cases}
(2\pi^2f_0^2(t-t_0)^2-1)exp(-\pi^2 f_0^2(t-t_0)^2), & \mbox{ if } 0\le t\le 2t_0,\\
0, & \mbox{otherwise}.
\end{cases}
\]
{\  See Figure \ref{ricker}. }

\begin{figure}[t]
\begin{center}
\includegraphics[scale=0.3]{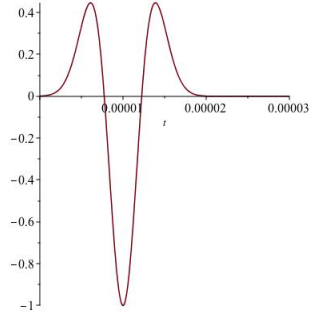}
\caption{Ricker wavelet $g(t)$}
\label{ricker}
\end{center}
\end{figure}
\noindent The spectrum content of $g(t)$ is visualized by its Fourier transform $\mathcal{F}\{g\}(\omega)$, {  see Figure \ref{real_spec} and \ref{imag_spec}}. Since the real part and the imaginary part is symmetric and anti-symmetric with respect to $\omega=0$, respectively, we only plot the $\omega\ge 0$ part of the graphs. 
\begin{figure}[b]
\centering
        \begin{subfigure}[h]{0.49\textwidth}
        \centering
                \includegraphics[width=\textwidth]{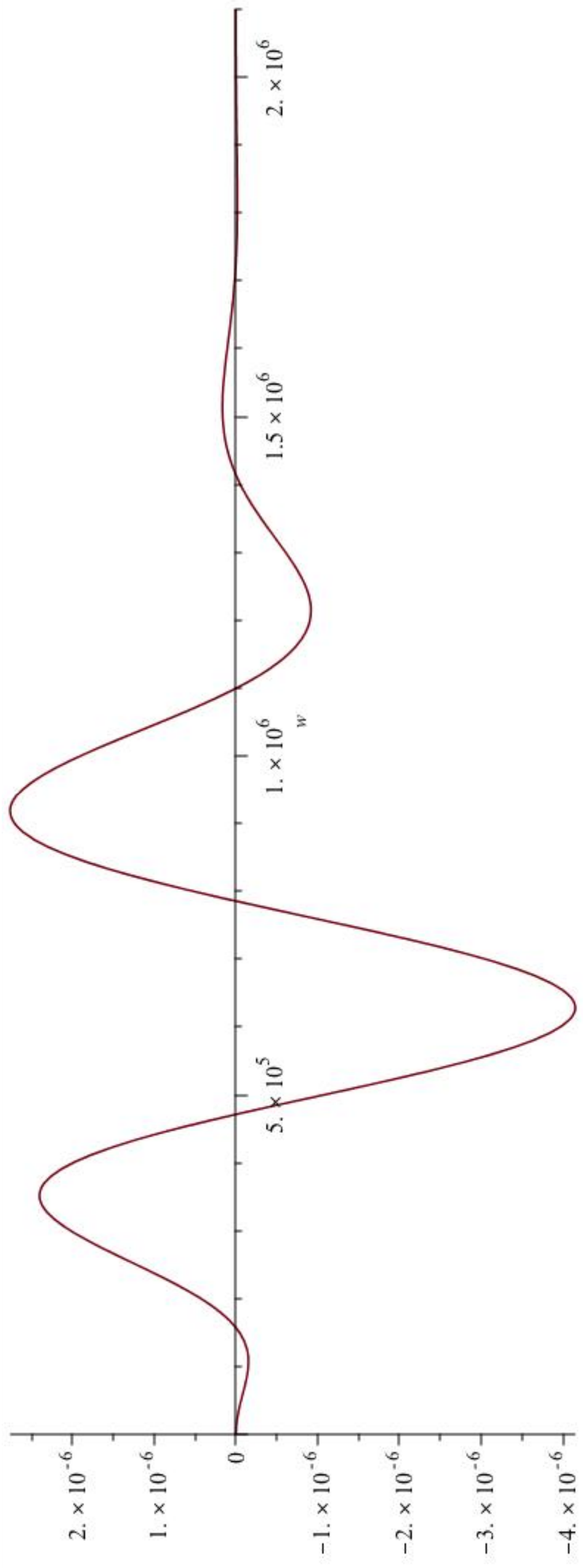}
                \caption{Real part of $\mathcal{F}\{g\}(\omega)$}
                \label{real_spec}
        \end{subfigure}
       \begin{subfigure}[h]{0.49\textwidth}
        \centering
                \includegraphics[width=\textwidth]{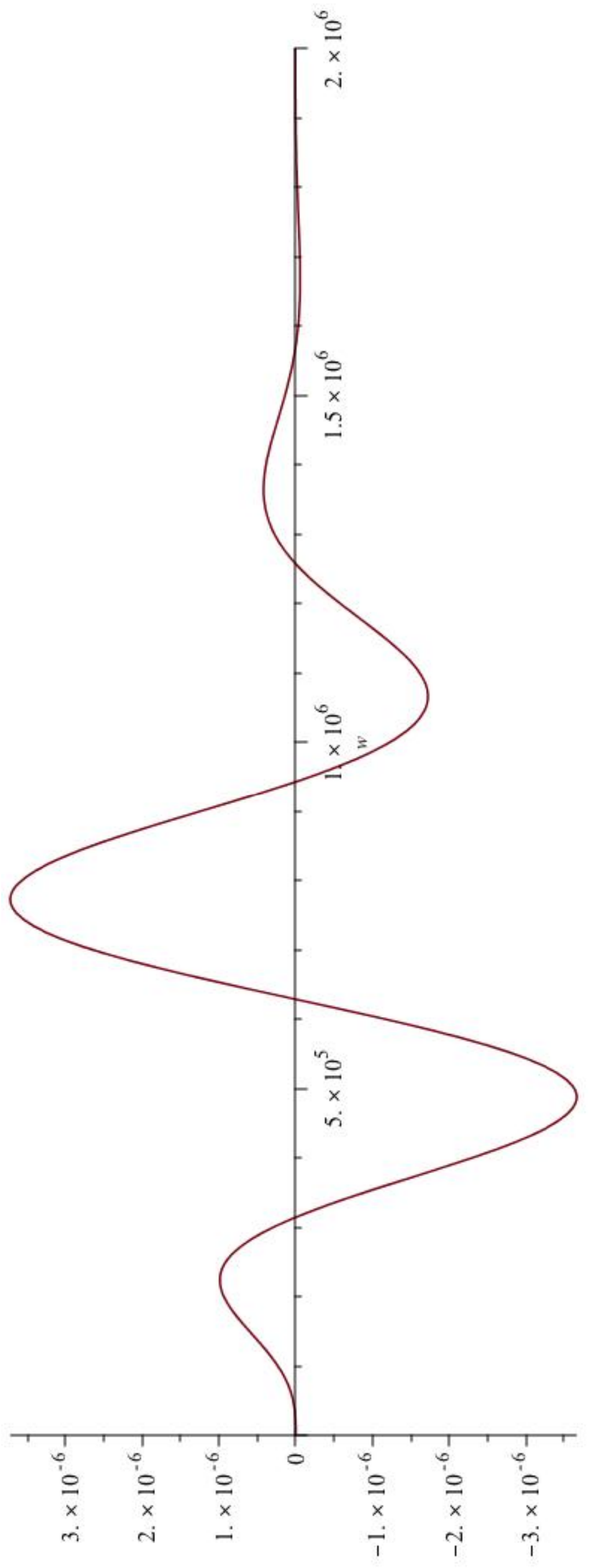}
                \caption{Imag. part of $\mathcal{F}\{g\}(\omega)$}
                \label{imag_spec}
        \end{subfigure}
\caption{Spectral content of $g(t)$}
\end{figure}
Based on Figure \ref{real_spec} and Figure \ref{imag_spec}, we choose the frequency range in our numerical simulation to be from $10^{-3}$ Hz to $2\times 10^{6}$ Hz. 

\begin{table}[htp]
\caption{Biot-JKD parameters}
\begin{center}
\begin{tabular}{|c|c|c|c|c|c|c|}
\hline
&&S1 & S2 & S3 &S4 & S5 \\\hline
$\rho_f (Kg\cdot m^{-3})$& \mbox{pore fluid density}&1000 &1040 &1040 &1040&1040\\
\hline
$\phi$(dimensionless)& \mbox{porosity}&0.8 &0.2&0.2&0.2&0.2\\
\hline
$\alpha_\infty$(dimensionless)&\mbox{infinite-frequency tortuosity}&1.1&3.6&2.0&2.0&3.6\\
\hline
$K_0 (m^2)$&\mbox{static permeability} &3e-8 & 1e-13& 6e-13&6e-13 &1e-13\\
\hline
$\nu (m^2\cdot s^{-1})$&\mbox{kinematic viscosity of pore fluid}& 1e-3/$\rho_f$ & 1e-3/$\rho_f$& 1e-3/$\rho_f$& 1e-3/$\rho_f$& 1e-3/$\rho_f$\\
\hline
$\Lambda (m)$&\mbox{structure constant}&2.454e-5&3.790e-6&6.930e-6&2.190e-7&1.20e-7\\
\hline
\end{tabular}
\end{center}
\label{parameters}
\end{table}%
We consider first the equal spaced sample points. Similar to what was reported in \cite{ou2014on-reconstructi}, the relative error peaked near low frequency. This is due to fact that in general, the function $D(s=-i\omega)$ varies the most near the lower end of $\omega$. This observation leads to the log-distributed grid points, which in general performs better in terms of maximum relative errors but with more ill-conditioned matrices.  For both the equally spaced and the log-spaced grids point, ill-conditioned matrices are involved. The ill-conditioning nature of the matrices $A$ in {  Approach} 1 and $S_1$, $S_2$ in {  Approach} 2, together with the fact there is no obvious preconditioner available for these matrices, we resort to the multiprecision package \emph{Advanpix}  \cite{mct2015} for directly solving \eqref{Algo_1} and the subsequent partial fraction decomposition involved in {  Approach} 1 and for solving the generalized eigenvalue problem \eqref{Algo_2}. These real-valued poles and residues are then converted to double precision before we evaluate the relative errors
\[
rel\_{err}(s):=\frac{|D^J(s)-D^J_{est}(s)|}{|D^J(s)|}, 
\]
where the pole-residue approximation function $D^J_{est}$ is defined in \eqref{the_point}. {  We set the number of significant digits in \emph{Advanpix} to be $90$, which is much higher than the $15$ decimal digits a 64-bit double-precision floating point format can represent.}
\begin{figure}[thb]
\centering
        \begin{subfigure}[thb]{0.95\textwidth}
        \centering
                \includegraphics[scale=0.4]{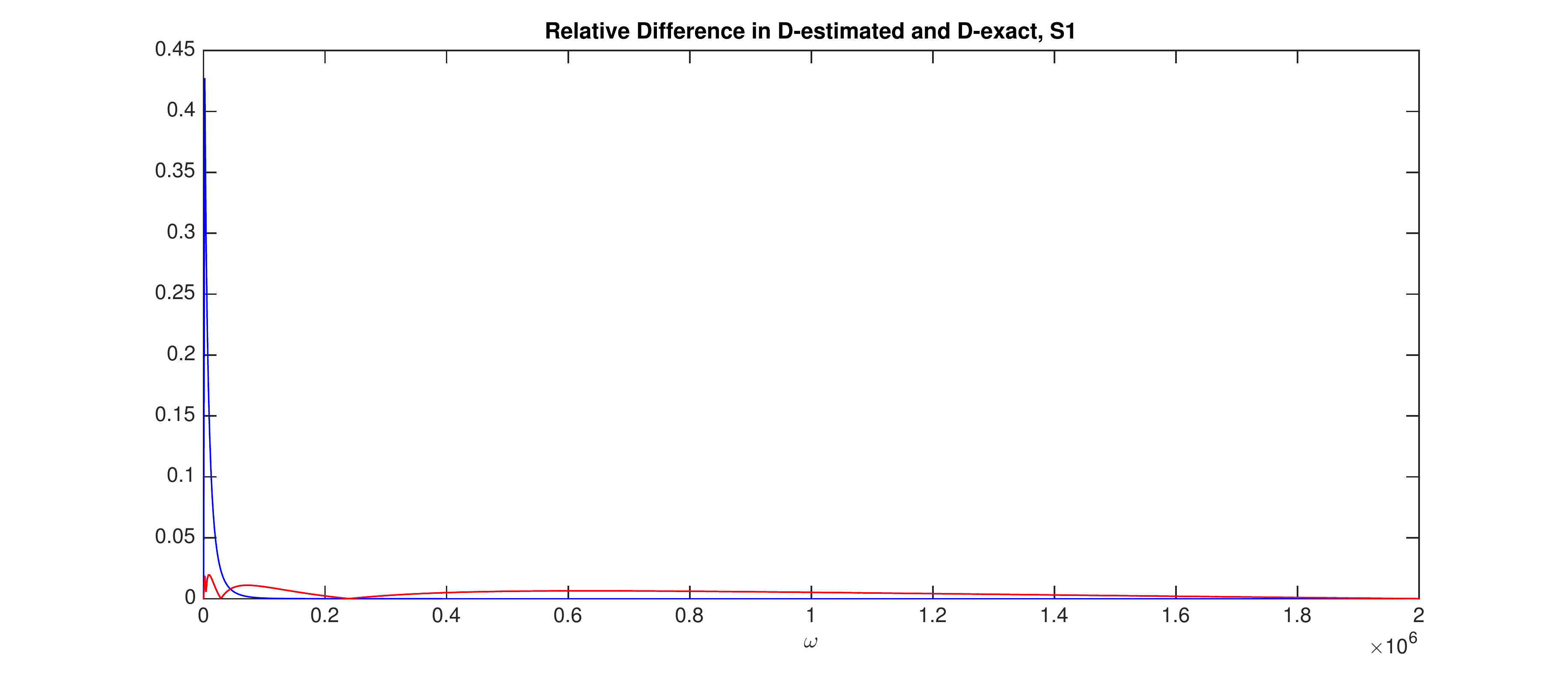}
                \caption{S1}
                \label{S1}
        \end{subfigure}
\\        
       \begin{subfigure}[thb]{0.95\textwidth}
        \centering
                \includegraphics[scale=0.4]{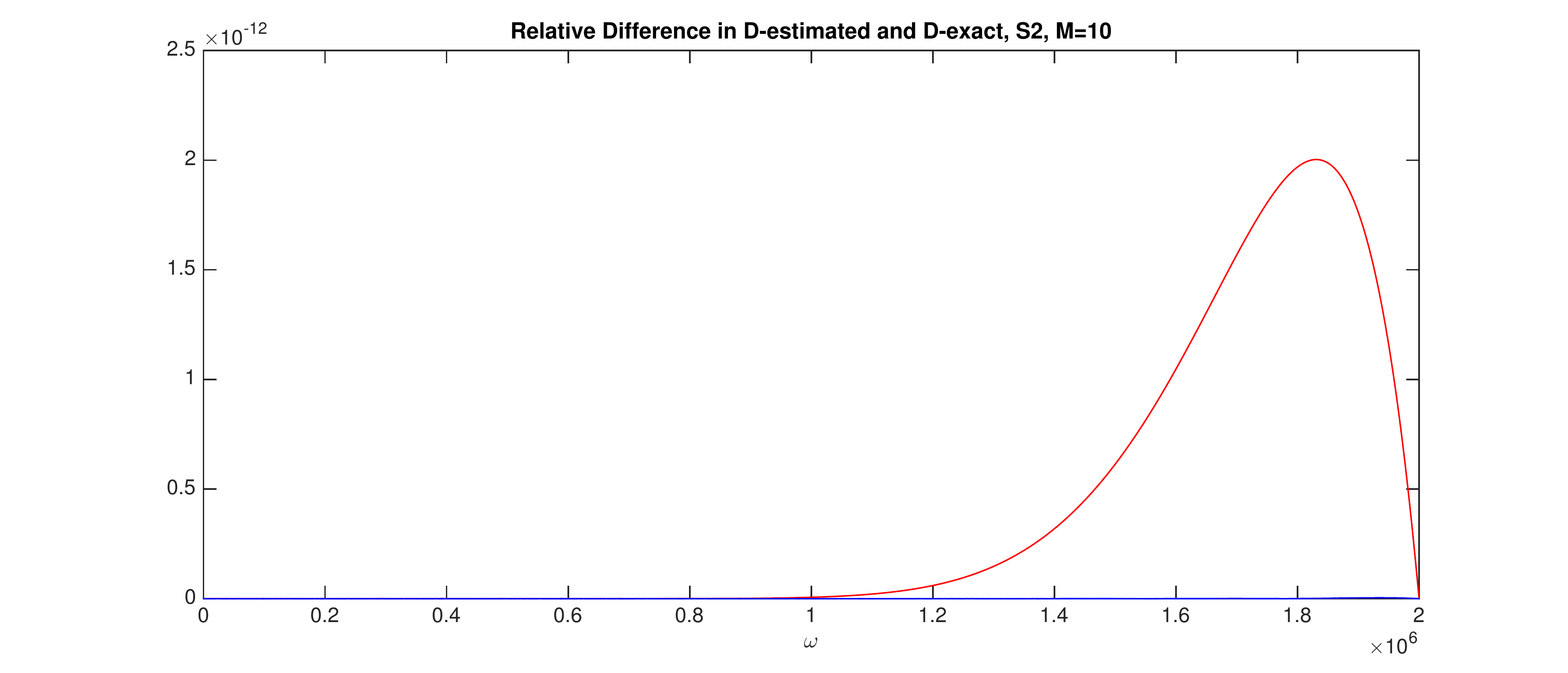}
                \caption{S2}
                \label{S2}
        \end{subfigure}
  \\  
       \begin{subfigure}[thb]{0.95\textwidth}
        \centering
                \includegraphics[scale=0.4]{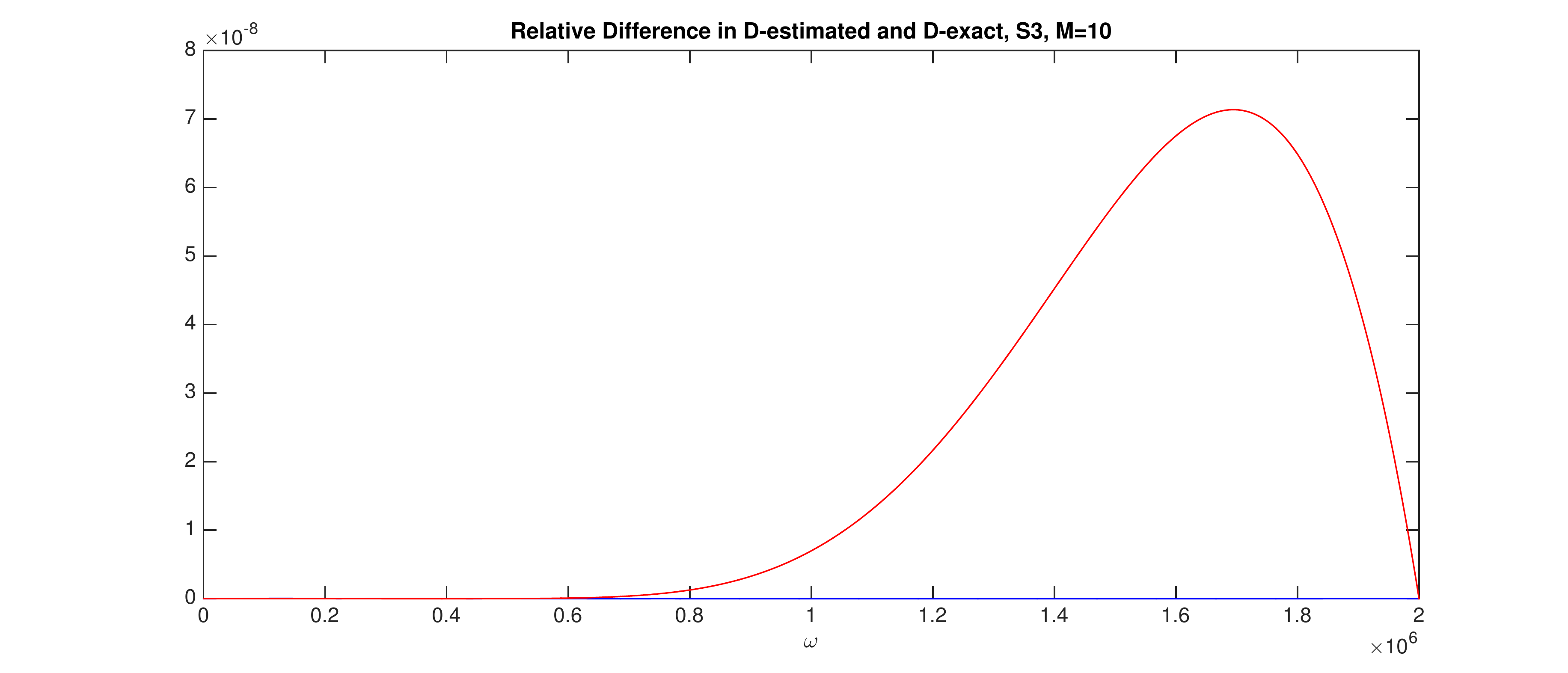}
                \caption{S3}
                \label{S3}
        \end{subfigure}
 \end{figure}

\begin{figure}[thb] \ContinuedFloat
\centering  
       \begin{subfigure}[thb]{0.95\textwidth}
        \centering
                \includegraphics[scale=0.4]{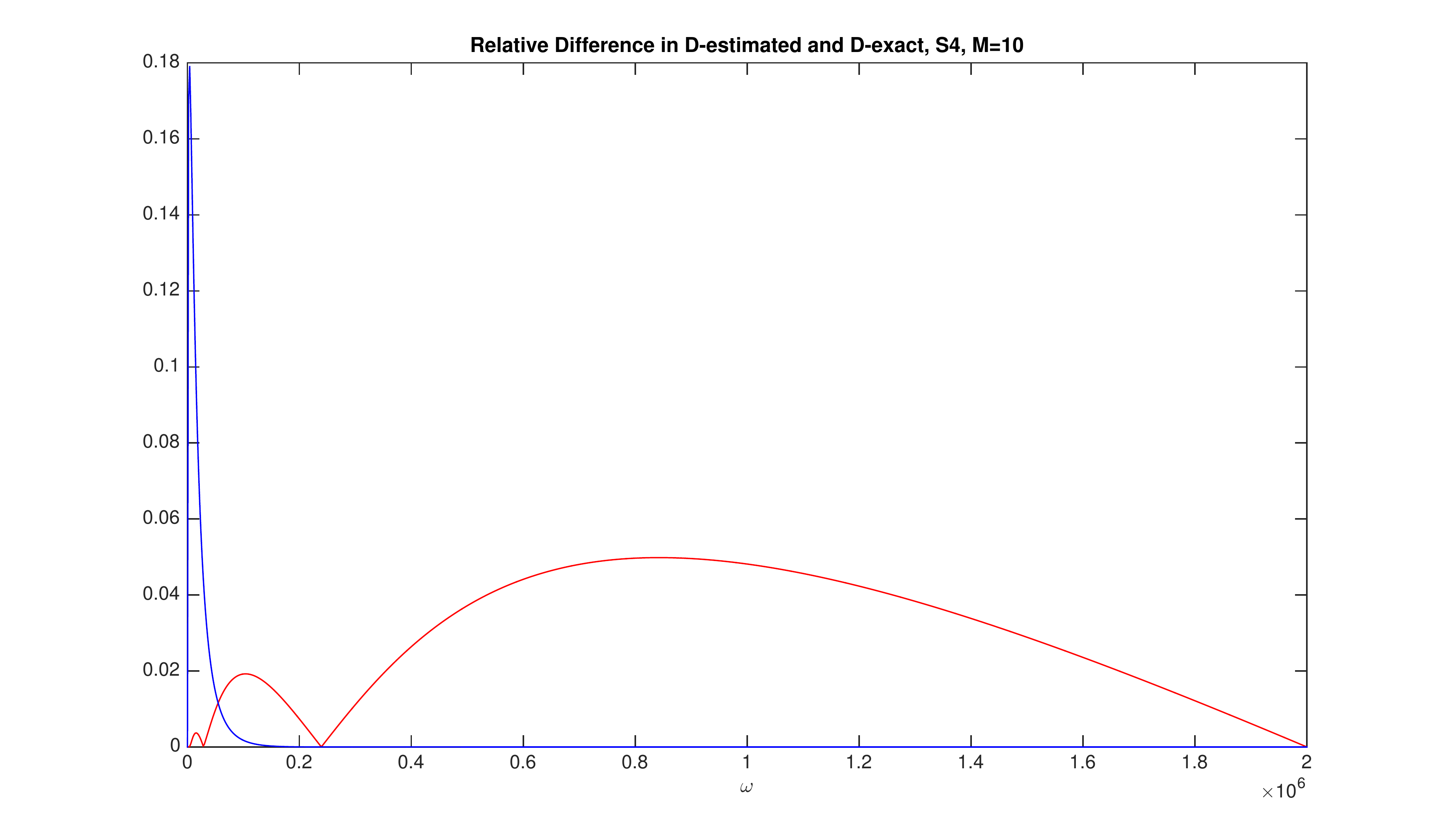}
                \caption{S4}
                \label{S4}
        \end{subfigure}
\\        
       \begin{subfigure}[thb]{0.95\textwidth}
        \centering
                \includegraphics[scale=0.4]{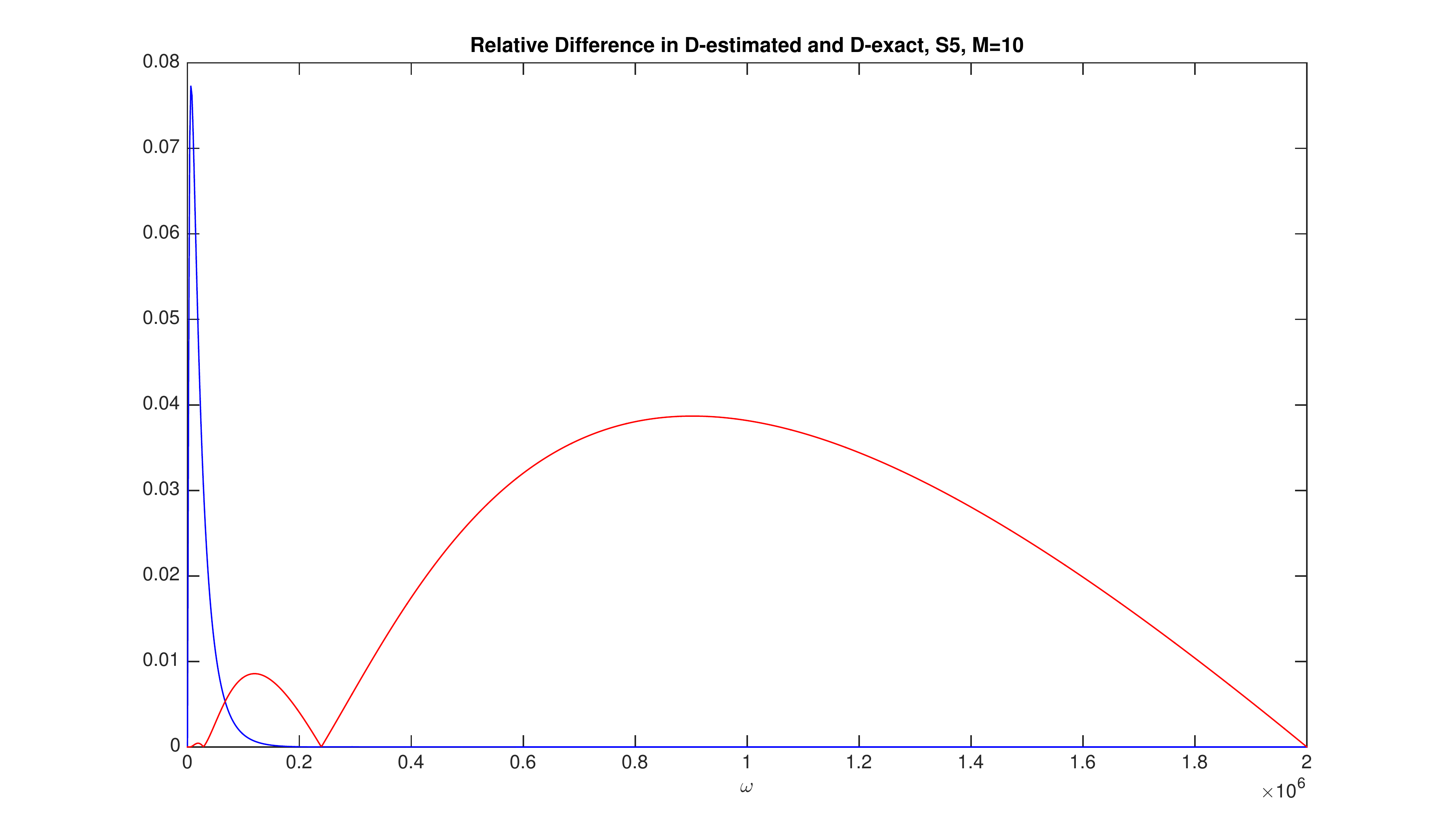}
                \caption{S5}
                \label{S5}
        \end{subfigure}
      \caption{Comparison of relative errors with $M=10$ for S1 to S5. Blue: Equally spaced grids, Red: log-distributed grids}
\end{figure}
The relative error $rel\_err$ with $M=10$ for all the 5 media listed in Table \ref{parameters} is plotted in Figure \ref{S1} to Figure \ref{S5}. The results by using equally space grids are in color blue while those by using log-distributed ones are in color red. 

Among all the 5 media listed in Table \ref{parameters}, the cancellous bone S1 and the sandstones S4 and S5 are the most difficult one to approximate in the sense that it requires the largest $M$ for achieving the same level of accuracy as for other media. {  The dynamic tortuosity functions of S1, S4 and S5 have large variation near low frequency and hence can be approximated much better when log-distributed grids are applied in the approximations. See Figure \ref{M14_equal} and Figure \ref{M14_log}.}

In {   Tables \ref{conditioning_S1} to \ref{conditioning_S5}}, we list the condition numbers for both of the equally-spaced grid points and the log-distributed one.  As can be seen, the condition numbers for matrices involved in {   Approach} 1 with log-distributed grid points worsen very rapidly with the increase of $M$ and the rescaling of volumes of $A$ is not effective when compared with the equally spaced case. {   In Figures \ref{scatter_M14_S1} to  \ref{scatter_M14_S5} }, where the poles and residues for {   $M=14$} computed with different combinations of methods are plotted in log-log scale, we see that {   Approach} 1 and 2 indeed give numerically identical results {   for all these 5 test media when the significant digits in the calculations are set much higher than the $\log_{10}$-scale of the condition numbers involved.} The calculation is carried out by using {$140$} significant digits and it takes about 5 seconds with a single processor MacBook Pro.

\begin{table}[h]
\caption{$\log_{10}$-Condition numbers of the matrices for material S1}
\begin{center}
\begin{tabular}{|c|c|c|c|c|}
\hline
&M=8& M=8 & M=14& M=14\\ \hline
&Equally spaced & log-spaced & Equally spaced & log-spaced\\ \hline
$A$& 48.9750&49.2328&86.7853&87.4283\\ \hline
$B$& 10.0354&31.4021& 15.2592&57.8602\\ \hline
$S_1$&13.3237&4.3014&23.7936&5.8767\\ \hline
$S_2$&9.9319&10.8471&30.3962&11.9522\\ \hline
\end{tabular}
\end{center}
\label{conditioning_S1}
\end{table}%
\begin{table}[h]
\caption{$\log_{10}$-Condition numbers of the matrices for material S2}
\begin{center}
\begin{tabular}{|c|c|c|c|c|}
\hline
&M=8& M=8 & M=14& M=14\\ \hline
&Equally spaced & log-spaced & Equally spaced & log-spaced\\ \hline
$A$&50.3092 &51.1536&88.1203&90.5197\\ \hline
$B$& 20.1795&64.8793&31.3151&120.2955 \\ \hline
$S_1$&15.1483&55.9887&29.1822&110.3533\\ \hline
$S_2$&15.1657&55.9766&29.1998&110.3456 \\ \hline
\end{tabular}
\end{center}
\label{conditioning_S2}
\end{table}%
\begin{table}[h]
\caption{$\log_{10}$-Condition numbers of the matrices for material S3}
\begin{center}
\begin{tabular}{|c|c|c|c|c|}
\hline
&M=8& M=8 & M=14& M=14\\ \hline
&Equally spaced & log-spaced & Equally spaced & log-spaced\\ \hline
$A$&49.9052&50.8287& 87.7160& 90.1957   \\ \hline
$B$&16.5867& 59.7475& 24.6405& 110.6680\\ \hline
$S_1$&12.3521&49.2816& 24.0606& 97.5660   \\ \hline
$S_2$&12.5097&49.2655& 24.2122& 97.5514   \\ \hline
\end{tabular}
\end{center}
\label{conditioning_S3}
\end{table}%
\begin{table}[h]
\caption{$\log_{10}$-Condition numbers of the matrices for material S4}
\begin{center}
\begin{tabular}{|c|c|c|c|c|}
\hline
&M=8& M=8 & M=14& M=14\\ \hline
&Equally spaced & log-spaced & Equally spaced & log-spaced\\ \hline
$A$&50.8209& 51.4408& 88.6312& 89.5011 \\ \hline
$B$&11.8964& 41.4130& 16.7579& 76.4563 \\ \hline
$S_1$&11.5699&18.8133& 22.0503&38.7220  \\ \hline
$S_2$& 14.6234& 18.7833&25.0888& 38.6911\\ \hline
\end{tabular}
\end{center}
\label{conditioning_S4}
\end{table}%
\begin{table}[h]
\caption{$\log_{10}$-Condition numbers of the matrices for material S5}
\begin{center}
\begin{tabular}{|c|c|c|c|c|}
\hline
&M=8& M=8 & M=14& M=14\\ \hline
&Equally spaced & log-spaced & Equally spaced & log-spaced\\ \hline
$A$&51.0851&51.8269&88.8954& 89.9913 \\ \hline
$B$& 12.2279&43.8165& 17.1587& 80.9135  \\ \hline
$S_1$& 11.3409&23.0730& 21.8547& 47.0251 \\ \hline
$S_2$& 13.8633&  23.0477& 24.3312& 46.9971    \\ \hline
\end{tabular}
\end{center}
\label{conditioning_S5}
\end{table}%
\begin{figure}[h]
\begin{center}
\includegraphics[scale=0.5]{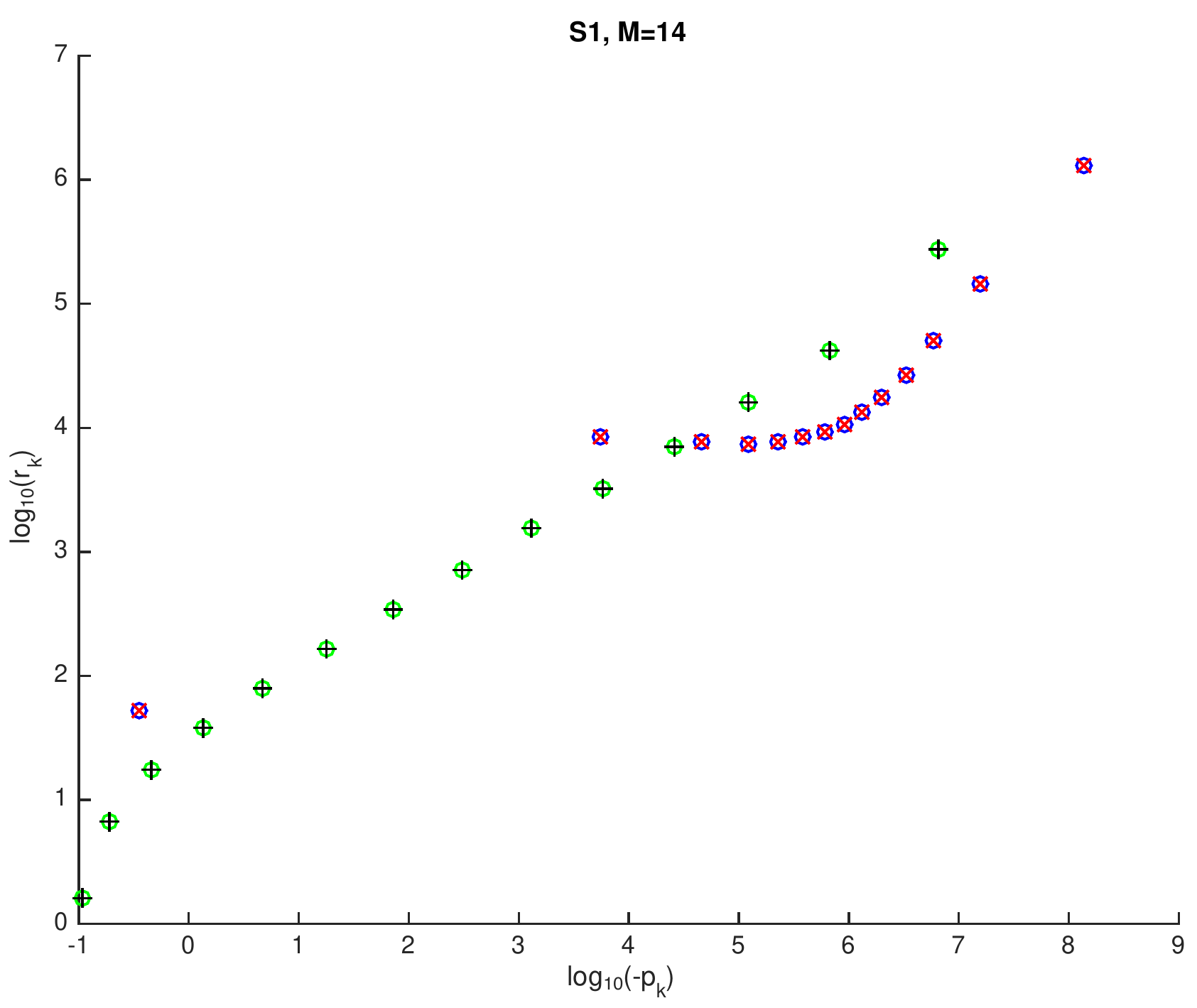}
\caption{($log_{10}(-p_k)$, $log_{10}(r_k)$), $k=1,..,14$. Blue circle: Approach 1 with Equally-spaced grids, Red x: Approach 2 with Equally-spaced grid, Green circle: Approach 1 with Log-spaced grids, Black +: Approach 2 with Log-spaced grids}
\label{scatter_M14_S1}
\end{center}
\end{figure}
\begin{figure}[h]
\begin{center}
\includegraphics[scale=0.5]{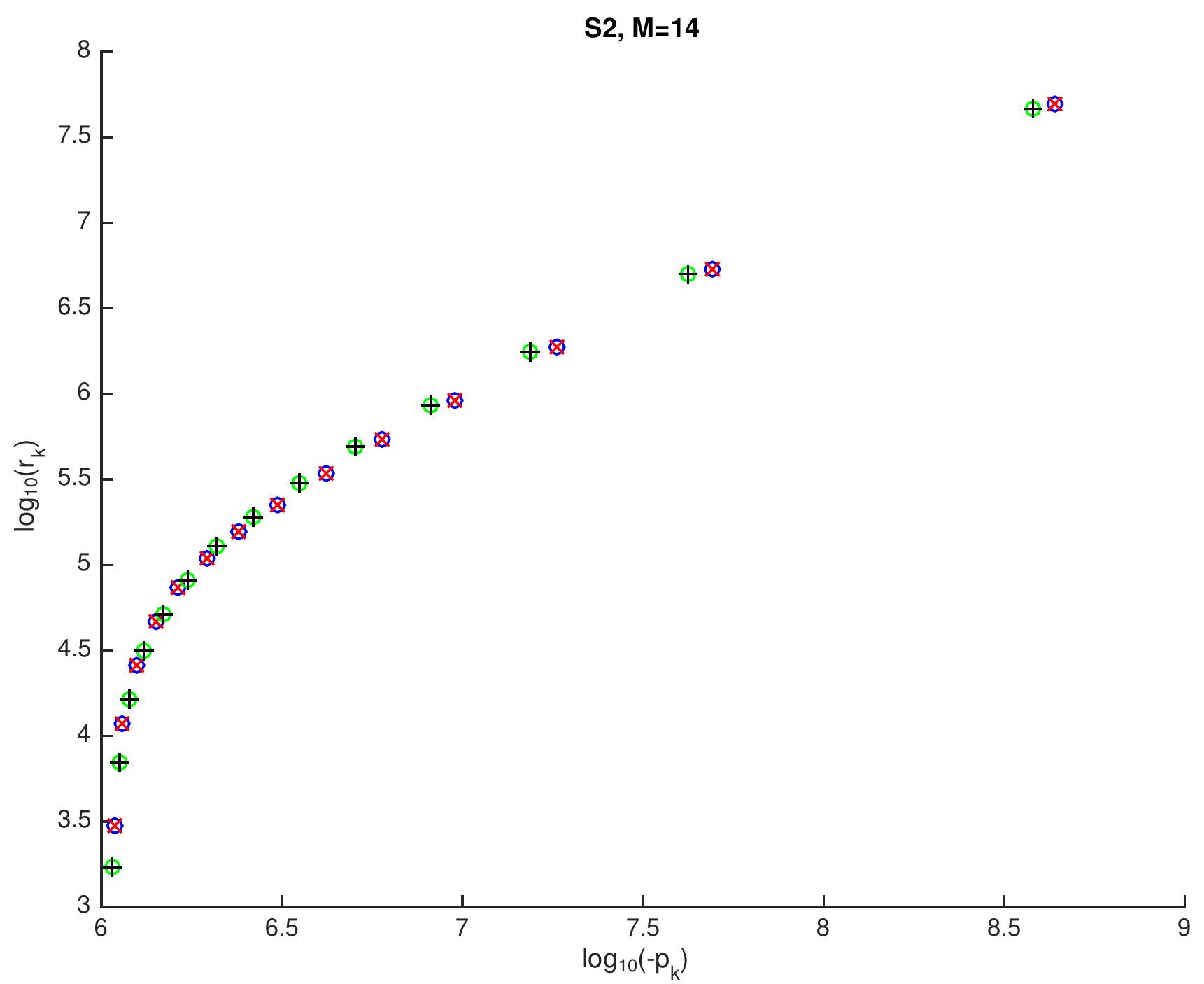}
\caption{($log_{10}(-p_k)$, $log_{10}(r_k)$), $k=1,..,14$. Blue circle: Approach 1 with Equally-spaced grids, Red x: Approach 2 with Equally-spaced grid, Green circle: Approach 1 with Log-spaced grids, Black +: Approach 2 with Log-spaced grids}
\label{scatter_M14_S2}
\end{center}
\end{figure}
\begin{figure}[h]
\begin{center}
\includegraphics[scale=0.5]{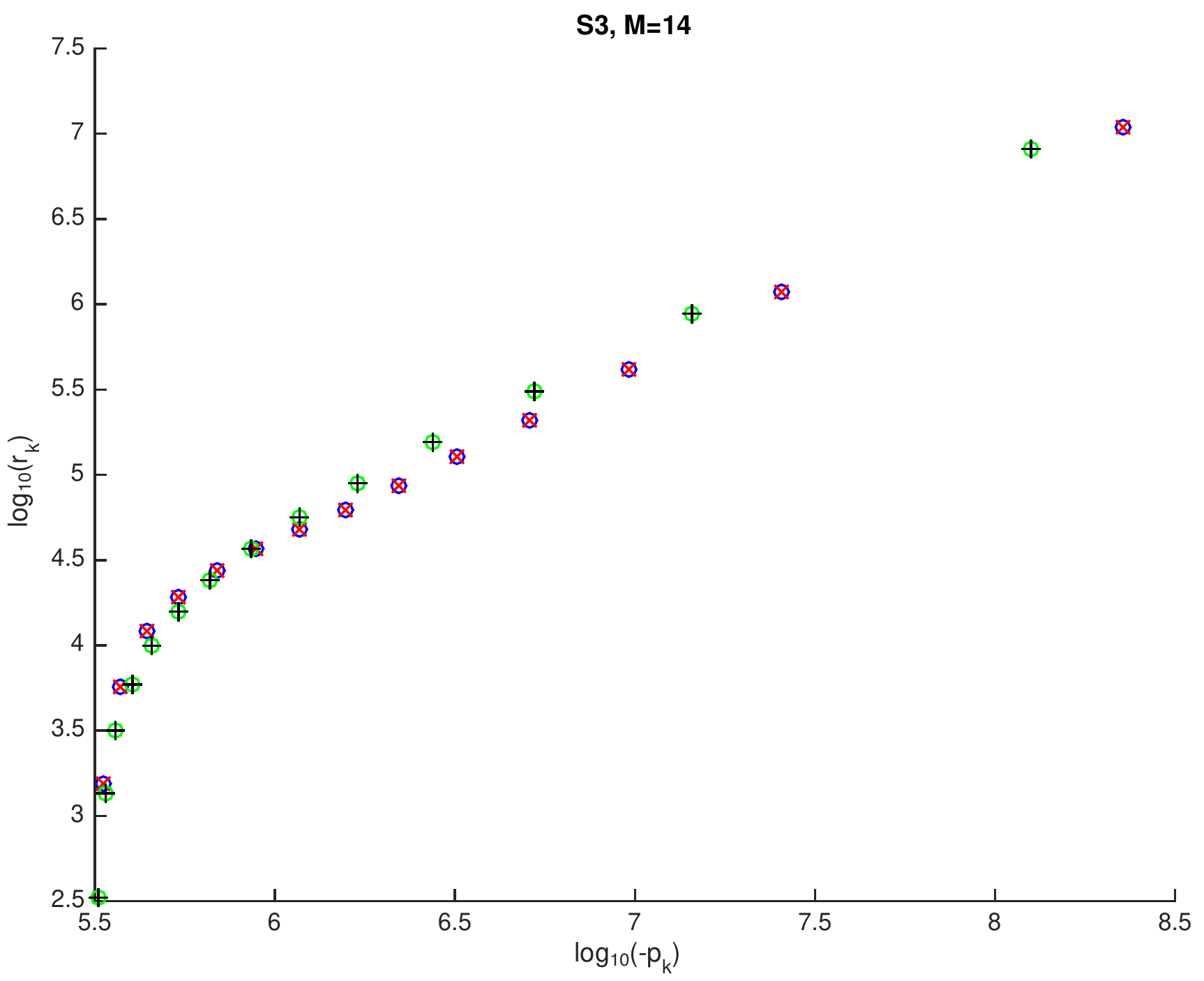}
\caption{($log_{10}(-p_k)$, $log_{10}(r_k)$), $k=1,..,14$. Blue circle: Approach 1 with Equally-spaced grids, Red x: Approach 2 with Equally-spaced grid, Green circle: Approach 1 with Log-spaced grids, Black +: Approach 2 with Log-spaced grids}
\label{scatter_M14_S3}
\end{center}
\end{figure}
\begin{figure}[h]
\begin{center}
\includegraphics[scale=0.5]{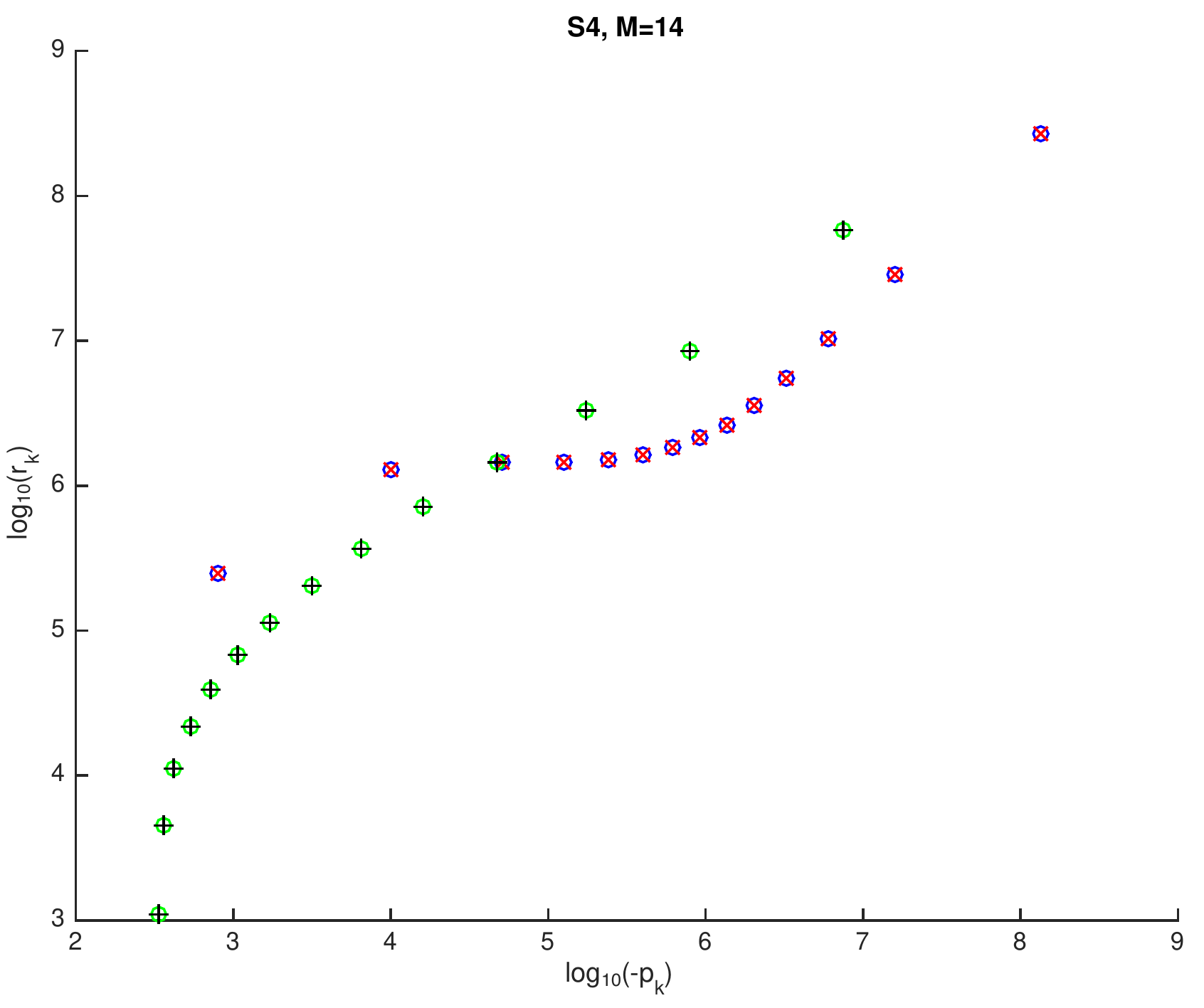}
\caption{($log_{10}(-p_k)$, $log_{10}(r_k)$), $k=1,..,14$. Blue circle: Approach 1 with Equally-spaced grids, Red x: Approach 2 with Equally-spaced grid, Green circle: Approach 1 with Log-spaced grids, Black +: Approach 2 with Log-spaced grids}
\label{scatter_M14_S4}
\end{center}
\end{figure}
\begin{figure}[h]
\begin{center}
\includegraphics[scale=0.5]{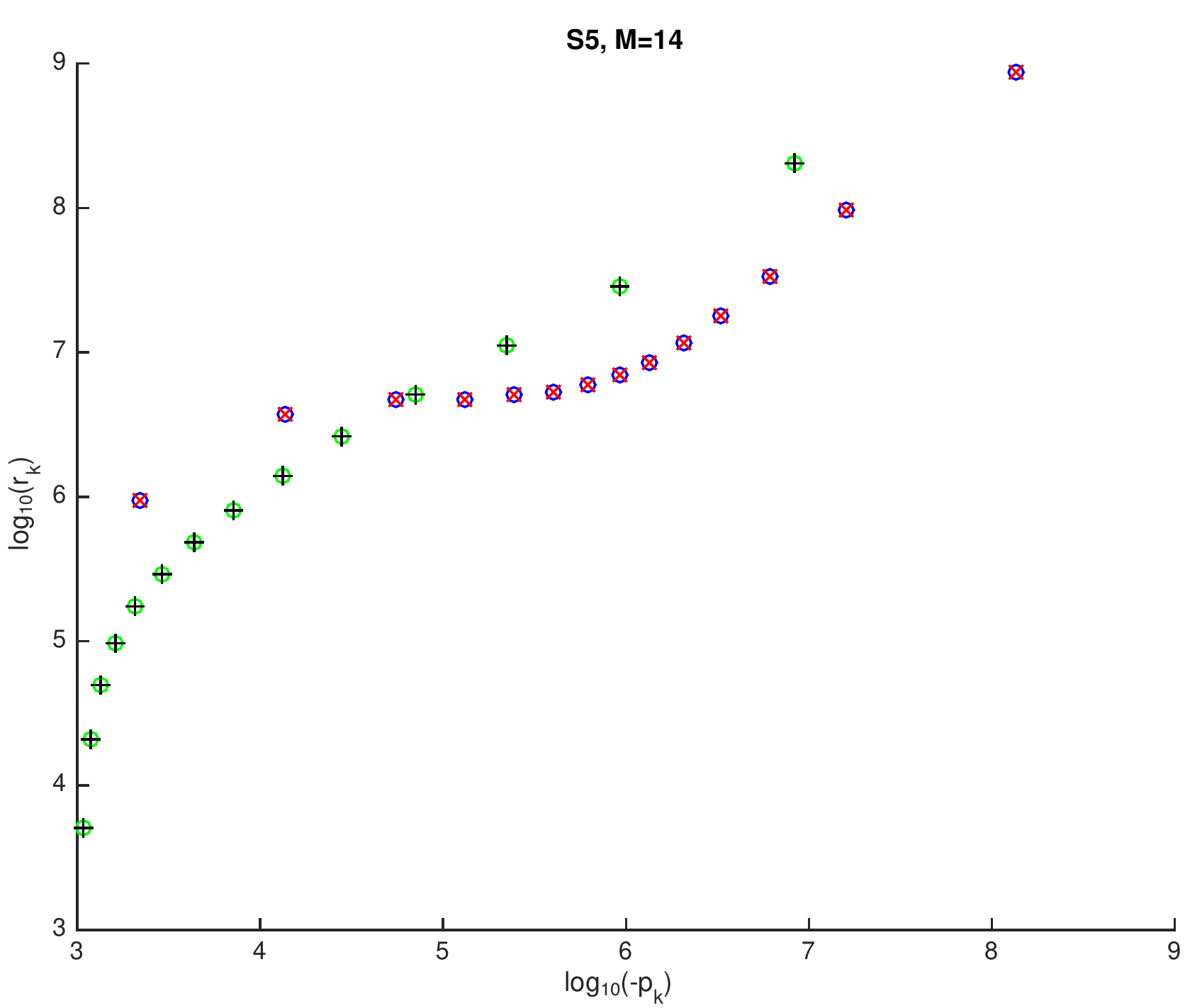}
\caption{($log_{10}(-p_k)$, $log_{10}(r_k)$), $k=1,..,14$. Blue circle: Approach 1 with Equally-spaced grids, Red x: Approach 2 with Equally-spaced grid, Green circle: Approach 1 with Log-spaced grids, Black +: Approach 2 with Log-spaced grids}
\label{scatter_M14_S5}
\end{center}
\end{figure}
%
\begin{figure}[h]
\begin{center}
\includegraphics[scale=0.5]{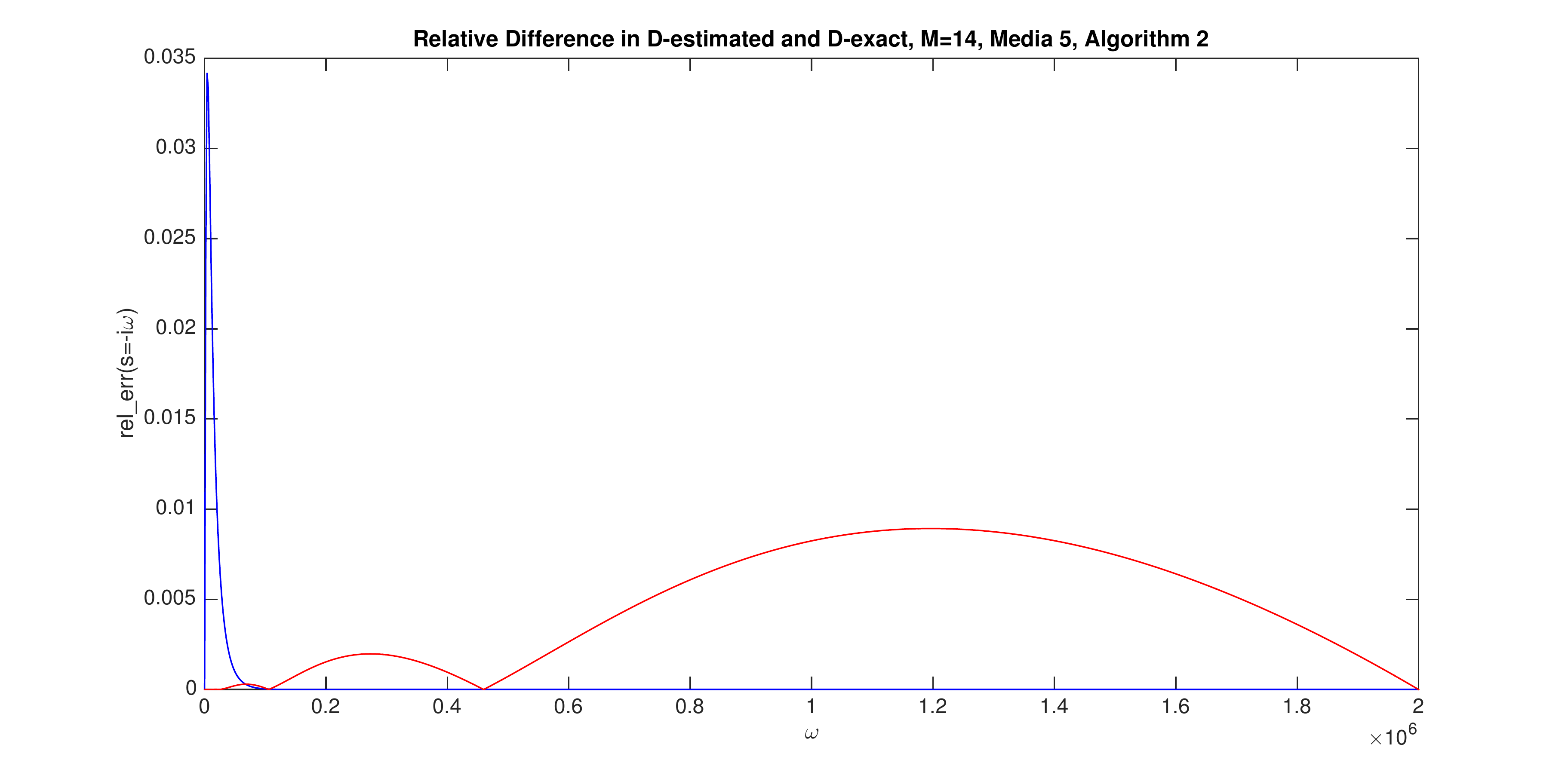}
\caption{$rel\_err(s=-i\omega)$, Blue: equally spaced grids, Red:  log-distributed grids}
\label{M14}
\end{center}
\end{figure}
In Figure \ref{M14}, the relative error $rel\_err$ for approximations by using equally spaced grid and by log-distributed grids are presented. As can be see from Figure \ref{M14_equal}, the peak of error near the lower end of the frequency range is due to the fact that the function being approximated needs more grid points there to resolve the variation. This is achieved by using the log-distributed grids. In Figure \ref{M14_equal} and Figure \ref{M14_log}, we plot $D^J$ and its pole-residue approximation $D^J_{est}$ to visualize the performance. Figure \ref{M14_equal} corresponds to the equally spaced grids while Figure \ref{M14_log} to the log-distributed one. In both figures, these two functions are almost indiscernible except the imaginary parts in Figure \ref{M14_equal} near the lower end of frequency where $rel\_{err}$ peaks; both the colors black (imaginary part of $D^J$) and green (imaginary part of $D^J_{est}$) can be seen there.
\begin{figure}[b]
\centering
        \begin{subfigure}[h]{0.95\textwidth}
        \centering
                \includegraphics[width=\textwidth]{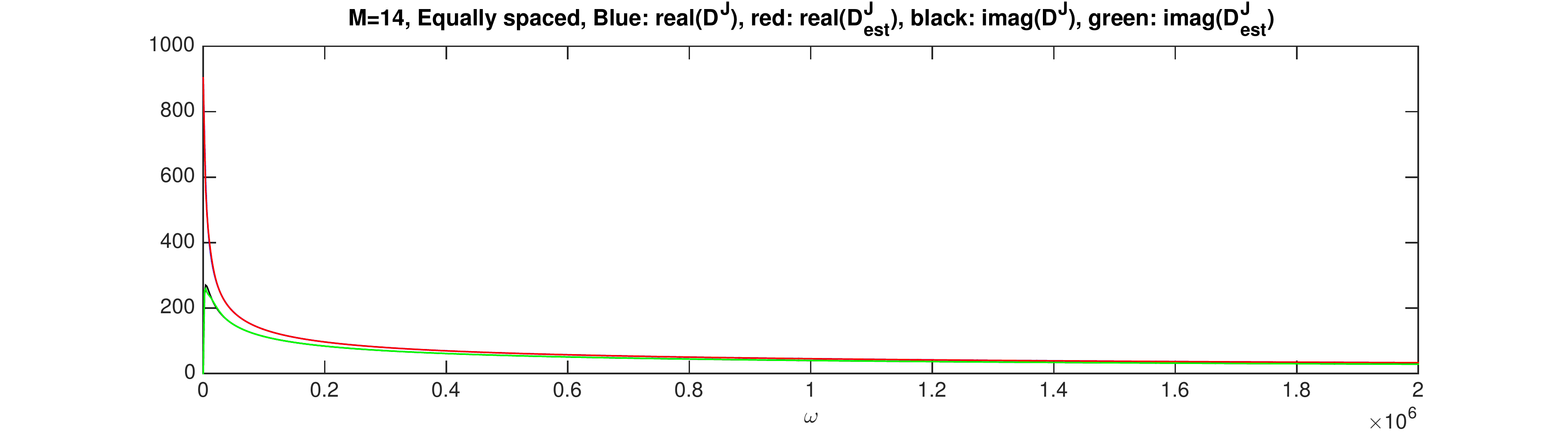}
                \caption{Equally spaced grids}
                \label{M14_equal}
        \end{subfigure}
\\        
       \begin{subfigure}[h]{0.95\textwidth}
        \centering
                \includegraphics[width=\textwidth]{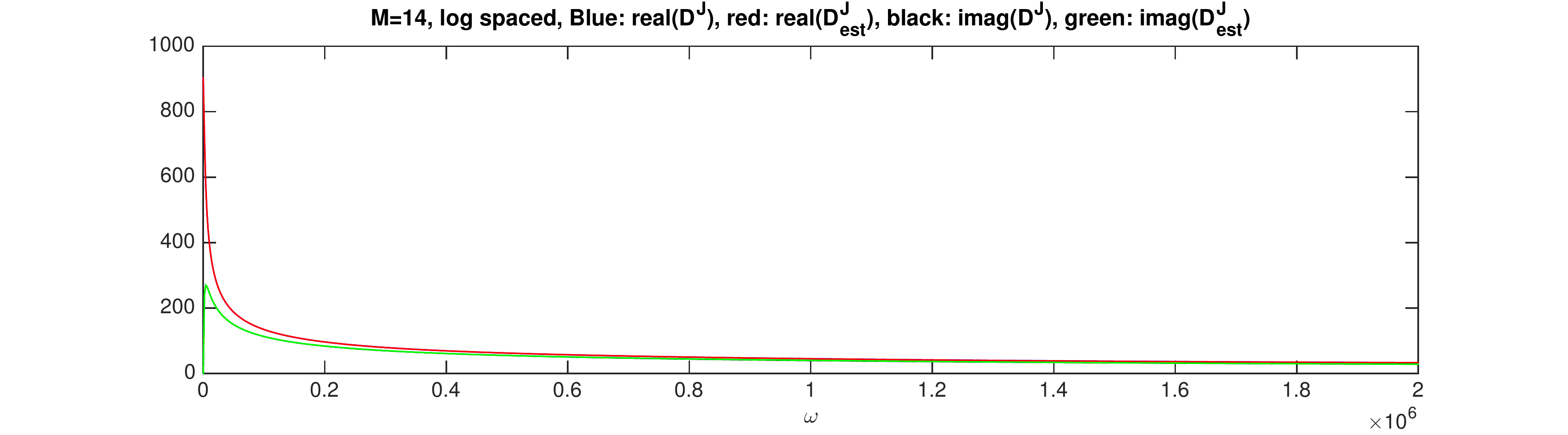}
                \caption{log-distributed grids }
                \label{M14_log}
        \end{subfigure}
\caption{Comparison of $D^J(s=-i\omega)$ and $D^J_{est}(s=-i\omega)$. Blue: $real(D^J)$, Red: $real(D^J_{est})$, Black: $imag(D^J)$, Green: $imag(D^J_{est})$}
\end{figure}
\section{Conclusions}
In this paper, we utilize the Stieltjes function structure of the JKD dynamic tortuosity to derive an augmented system of Biot-JKD equations  \eqref{aug_b}-\eqref{aug_e} that approximates the solution of the original Biot-JKD equations \eqref{test}-\eqref{convo}. Asymptotic behavior of the tortuosity function as $\omega\goto \infty$ is enforced analytically before the numerical interpolation carried out by {   Approach} 1 and {   Approach} 2. Due to the nature of the tortuosity functions{    of S1, S4 and S5}, log-distributed interpolation points generally perform better than the equally distributed ones. We tested our {   approach}s on 5 sets of poroelastic parameters obtaining from the existing literature and interpolated the JKD dynamic tortuosity equation to high accuracy through a frequency range that spans 9 orders of magnitude from $10^{-3}$ to $2\times 10^6$. The extremely ill-conditioned matrices are dealt with by using a multiprecision package \emph{Advanpix} in which we set the significant digits of floating numbers to be {   140}. It turns out {   approach}s 1 and 2 {   give numerically identical results for all the test cases when the significant digits are set much higher than the {$log_{10}$}-scale of the condition numbers involved}. We think the exact link between these two {   approach}s can be derived through the Barycentric forms for rational approximations \cite{berrut1997lebesgue-consta}, which in term provides an {   approach} that can adapt the choice of grid points based on the data points so the Lebesgue constant is minimized  \cite{nakatsukasa2016the-aaa-algorit}. This will be explored in a later work. 

\medskip

{\it Acknowledgments.} We wish to thank Joe Ball, Daniel Alpay, Marc van Barel, Thanos Antoulas, Sanda Lefteriu, and Cosmin Ionita for their helpful suggestions. 

\clearpage
\bibliographystyle{plain}

\end{document}